\documentclass[11pt,bezier]{article}
\usepackage{tikz}
\usetikzlibrary{arrows.meta,decorations.markings,angles,
calc, chains,
quotes,
positioning,
shapes.geometric,patterns,snakes}
\usetikzlibrary{arrows,shapes,trees}
\usepackage{amsmath,amssymb,amsfonts,euscript,graphicx}
\usepackage{float}
\usepackage{caption}
\usepackage{bigstrut}
\usepackage{tabu}

\newtheorem{preproof}{{\bf \indent Proof.}}

\newenvironment{proof}[1]{\begin{preproof}{\rm
               #1}\hfill{$\Box$}}{\end{preproof}}

\newtheorem{example}{\bf\indent Example}[section]
\newtheorem{thm}{{\bf\indent Theorem}}[section]

\newtheorem{lem}{\bf\indent Lemma}[section]

\title{\bf \large  Sombor index of clean graphs\thanks
{{\it Key Words}: Clean graph, Sombor index, Vertex degree\newline
{\indent{~~2020 {\it Mathematics Subject Classification}: 13A15, 13C05.}}}}

\author{{\normalsize  {\sc M. Badie${}^{\mathsf{}}$,  R. Nikandish${}^{\mathsf{}}$\thanks{Corresponding author},     M. Pirniia${}^{\mathsf{}}$}
}\vspace{3mm}\\
{\footnotesize{${}^{\mathsf{}}$}}\\
{\footnotesize{${}^{\mathsf{}}$\it  Department of Mathematics,
Jundi-Shapur University of Technology,  Dezful,
Iran}}\\
{\footnotesize{${}^{\mathsf{}}$\it}}\\
{\footnotesize{$\mathsf{badie@jsu.ac.ir}$\quad\quad $\mathsf{r.nikandish@ipm.ir}$ \quad\quad
$\mathsf{mehdipirnya442@gmail.com}$\quad\quad}}}

\begin{document}

\maketitle

%%%%%%%%%%%%%%%%%%%%%%%%%%%%%%%%%%%%%%%%%%%%%%%%%%%%%%%%%%%%%%%%%%%%%%%%%%%%%%%%%%%%%%%%%%%%%%%%%%%%%%%%%%%%%%%%%%%%%%%%%%%%%%%%%%%%%%%%%%%%%%%%%%%%%%%
\begin{abstract}
{\small Let $G=(V,E)$ be a graph with the vertex set $V(G)$ and edge set $E(G)$. The Sombor index of $G$, denoted by $SO(G)$, is defined as $\sum_{uv\in E(G)} \sqrt{deg(u)^{2}+deg(v)^{2}}$, where $deg(u)$ is the degree of vertex $u$ in $V(G)$. The clean graph of a ring $R$, denoted by $Cl(R)$, is a graph with vertex set  $\{(e,u): e\in Id(R), u\in U(R)\}$ and two distinct vertices $(e,u)$ and $(f,v)$ are adjacent if and only if $ef=0$ or $uv=1$ ($Id(R)$ and $U(R)$ are the sets of idempotents and unit elements of $R$, respectively). The induced subgraph on $\{(e,u): e\in Id^\ast(R), u\in U(R)\}$ is denoted by $Cl_2(R)$. In this paper,  $SO(Cl_2(\mathbb {Z}_n))$, for different values of the positive integer $n$, is investigated.}
\end{abstract}
%%%%%%%%%%%%%%%%%%%%%%%%%%%%%%%%%%%%%%%%%%%%%%%%%%%%%%%%%%%%%%%%%%%%%%%%%%%%%%%%%%%%%%%%%%%%%%%%%%%%%%%%%%%%%%%%%%%%%%%%%%%%%%%%%%%%%%%%%%%%%%%%%%%%%%%
\begin{center}\section{Introduction}\end{center}
Throughout this article, all graphs $G = (V,E)$ are simple and connected. Let
$G$ be a graph with the vertex set $V(G)$ and the edge set $E(G)$. By $uv\in E(G)$, we denote the edge of $G$ connecting the vertices $u$ and $v$. A complete graph of order $n$ is denoted by $K_{n}$. Moreover, an induced subgraph on $S\subseteq V(G)$ is denoted by $G[S]$. The Sombor index of a graph $G$, $SO(G)$, is defined as $\sum_{uv\in E(G)} \sqrt{deg(u)^{2}+deg(v)^{2}}$, where $deg(u)$ is the degree of vertex $u$ in $V(G)$. The Sombor index is one of the most recent graph invariants which has been devised by Gutman in 2021 \cite{Gut} based on the degree of vertices. This interesting chemical graph theoretic concept has became the subject of hundreds papers in the period of 2021 to 2024, see for instance \cite{Ali, Das, Do, Shet, Wei}. These works motivated algebraic graph theorists to calculate the Sombor index in graphs associated with algebraic structures; Some example in this field may be found in \cite{gur, rat}. 

Throughout this paper, all rings are assumed to be commutative with identity. The set of all idempotent, unit and self-inverse elements  of $R$ are denoted by  $Id(R)$, $U(R)$ and $U'(R)$. Moreover, we let $U''(R)=U(R)\setminus U'(R)$. For a subset $T$ of a ring $R$ we let $T^*=T\setminus\{0\}$. 

The clean graph of a ring $R$, denoted by $Cl(R)$, is a graph with vertex set  $\{(e,u): e\in Id(R), u\in U(R)\}$ and two distinct vertices $(e,u)$ and $(f,v)$ are adjacent if and only if $ef=0$ or $uv=1$ ($Id(R)$ and $U(R)$ are the sets of idempotents and unit elements of $R$, respectively). The induces subgraph on $\{(e,u): e\in Id^\ast(R), u\in U(R)\}$ is denoted by $Cl_2(R)$. The idea of clean graph was first introduced by Habibi et. al. in \cite{hab} and many interesting properties of this graph such as clique and chromatic number, independence number and domination number were studies. The authors also showed that  $Cl_2(R)$ has a crucial role in this graph. Afterwards,  $Cl(R)$ and $Cl_2(R)$ attracted many attentions , for instance see \cite{mat, ram}.  In this paper, we computed $SO(Cl_2(\mathbb {Z}_n))$, for different values of the positive integer $n$.
%%%%%%%%%%%%%%%%%%%%%%%%%%%%%%%%%%%%%%%%%%%%%%%%%%%%%%%%%%%%%%%%%%%%%%%%%%%%%%%%%%%%%%%%%%%%%%%%%%%%%%%%%%%%%%%%%%%%%%%%%%%%%%%%%%%%%%%%%%%%%%%%%%%%%%%
%%%%%%%%%%%%%%%%%%%%%%%%%%%%%%%%%%%%%%%%%%%%%%%%%%%%%%%%%%%%%%%%%%%%%%%%%%%%%%%%%%%%%%%%%%%%%%%%%%%%%%%%%%%%%%%%%%%%%%%%%%%%%%%%%%%%%%%%%%%%%%%%%%%%%%%%
\section{Main Results}
 \noindent
 
%%%%%%%%%%%%%%%%%%%%%%%%%%%%%%%%%%%%%%%%%%%%%%%%%%%%%%%%%%%%%%%%%%%%%%%%%%%%%%%%%%%%%%%%%%%%%%%%%%%%%%%%%%%%%%%%%%%%%%%%%%%%%%%%%%%%%%%%%%%%%%%%%%%%%%%
%%%%%%%%%%%%%%%%%%%%%%%%%%%%%%%%%%%%%%%%%%%%%%%%%%%%%%%%%%%%%%%%%%%%%%%%%%%%%%%%%%%%%%%%%%%%%%%%%%%%%%%%%%%%%%%%%%%%%%%%%%%%%%%%%%%%%%%%%%%%%%%%%%%%%%%%

We recall the following two lemmas which determine the number of idempotent and self-inverse elements in the ring  $\mathbb{Z}_n$,  where $n=\prod\limits_{i=1}^kp_i^{\alpha_i}$, $p_i$'s are distinct primes and $\alpha_i$ are non-negative integers.
\begin{lem}\label{dim1}
	Let $n=\prod\limits_{i=1}^kp_i^{\alpha_i}$, where $p_i$'s are distinct primes. Then ‌ $\mathbb{Z}_n$ contains $2^k$ idempotents. 
\end{lem}

{See \cite{hew}.}

\begin{lem}\label{dim2}
	Suppose that $m$ is a non-negative integer, $n=2^m\prod\limits_{i=1}^kp_i^{\alpha_i}$ and
	$p_i$'s$\neq 2$ are distinct primes. Then 
$ |U'(\mathbb{Z}_n)|=\begin{cases}2^k & \text{if} ~ m\in\{0,1\},	\\
	2^{k+1} & \text{if}~m=2, \\
	2^{k+2} & \text{if}~ m\geq3.

	\end{cases}$
\end{lem}
\begin{proof}
{See \cite[Proposition 3.1]{sin}.}

\end{proof}
The next three lemmas calculate $SO(Cl_2(\mathbb {Z}_n))$, where $n \in \{p^\alpha, p^\alpha q^\beta\}$.
\begin{lem}\label{dimfinite}
Let $n=p^\alpha$, where $p\neq 2$ is prime and $\alpha$ is a positive integer. Then $SO(Cl_2(\mathbb {Z}_n))=\frac{\phi(n)-2}{2}\sqrt{2}$.
\end{lem}
\begin{proof}
{By Lemma \ref{dim1}, $|Id(\mathbb {Z}_n)|=1.$ Also, by Lemma \ref{dim2}, $r=|U'(\mathbb {Z}_n)|=2.$ Thus, we may let, $U'(\mathbb {Z}_n)=\{1,u_2=u_r\}$. Hence $V(Cl_2(\mathbb {Z}_n))=\{(1,1), (1,u_2),\ldots, (1,\phi(n))\}.$ Since $(1,1), (1,u_2)$ are isolated vertices, there are $\phi(n)-2$ vertices of the form $(1,u)$ for which $u\in U''$. Hence for every $(1,u_i)\in V(Cl_2(\mathbb {Z}_n))$ there exists $(1,u_j)$ such that $u_iu_j=1$, for $i\neq j, i,j \geq3.$ Now, it is not hard to see that $Cl_2(\mathbb {Z}_n)=\frac{\phi(n)-2}{2}K_2 \cup 2K_1$. Therefore,  $SO(Cl_2(\mathbb {Z}_n))=\frac{\phi(n)-2}{2}\sqrt{2}$.}
\end{proof}
\begin{figure}[H]
\centering
\begin{tikzpicture}[line cap=round,line join=round,>=triangle 45,x=1.0cm,y=1.0cm]
\draw (6.,3.)-- (6.,2.);
\draw (7.,3.)-- (7.,2.);
\draw (9.,3.)-- (9.,2.);
\draw (7.8,2.66) node[anchor=north west] {$\cdots$};
\begin{scriptsize}
\draw [fill=black] (2.,3.) circle (1.5pt) node[below] {$(1,1)$};
\draw [fill=black] (4.,3.) circle (1.5pt) node[below] {$(1,u_2)$};
\draw [fill=black] (6.,3.) circle (1.5pt);
\draw [fill=black] (7.,3.) circle (1.5pt);
\draw [fill=black] (9.,3.) circle (1.5pt);
\draw [fill=black] (6.,2.) circle (1.5pt);
\draw [fill=black] (7.,2.) circle (1.5pt);
\draw [fill=black] (9.,2.) circle (1.5pt);
\draw[decoration={brace,mirror},decorate] (6,1.75) --(9,1.75) node[below,xshift=-1.5cm,yshift=-0.25cm] {$\dfrac{\phi(n)-2}{2}$};
\end{scriptsize}
\end{tikzpicture}
\caption{$Cl_2(\mathbb{Z}_{p^\alpha})$, $p\neq 2$}
\end{figure}
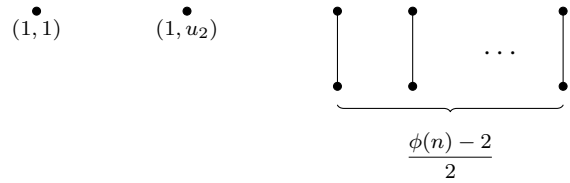
\begin{lem}\label{dimfi}
	Let $n=2^\alpha$, where  $\alpha$ is a positive integer. Then 
$SO(Cl_2(\mathbb {Z}_n))=\frac{\phi(n)-4}{2}\sqrt{2}$.
\end{lem}
\begin{proof}
{If $\alpha=1,2$, then there is nothing to prove. So let $\alpha\geq3$. By Lemma \ref{dim1}, $|Id(\mathbb {Z}_n)|=1.$ Also, by Lemma \ref{dim1}, $r=|U'(\mathbb {Z}_n)|=4.$ Indeed, $U'(\mathbb {Z}_n)=\{1, 2^{\alpha-1}-1, 2^{\alpha-1}+1, 2^{\alpha}-1\}$. Thus $V(Cl_2(\mathbb {Z}_n))=\{(1,1), (1,u_2),(1,u_3), (1,u_{r=4}),\ldots, (1,\phi(n))\}.$ Since $(1,1),\ldots, (1,u_4)$ are isolated vertices, there are $\phi(n)-4$ vertices of the form $(1,u)$ for which $u\in U''$. Hence for every $(1,u_i)\in V(Cl_2(\mathbb {Z}_n))$ there exists $(1,u_j)$ such that $u_iu_j=1$, for $i\neq j, 1\leq i,j \leq \phi(n).$ Now, it is not hard to see that $Cl_2(\mathbb {Z}_n)=\frac{\phi(n)-4}{2}K_2 \cup 4K_1$. Therefore,  $SO(Cl_2(\mathbb {Z}_n))=\frac{\phi(n)-4}{2}\sqrt{2}$.}
\end{proof}
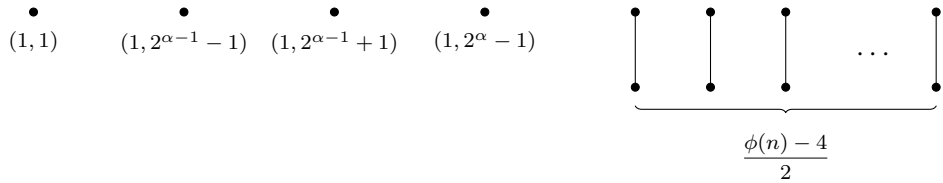
\begin{figure}[H]
\centering
\begin{tikzpicture}[line cap=round,line join=round,>=triangle 45,x=1.0cm,y=1.0cm]
\draw (5.,3.)-- (5.,2.);
\draw (7.,3.)-- (7.,2.);
\draw (9.,3.)-- (9.,2.);
\draw (7.8,2.66) node[anchor=north west] {$\cdots$};
\draw (6.,3.)-- (6.,2.);
\begin{scriptsize}
\draw [fill=black] (1.,3.) circle (1.5pt)  node[below,yshift=-0.15cm] {$(1,2^{\alpha-1}+1)$};
\draw [fill=black] (3.,3.) circle (1.5pt) node[below,yshift=-0.15cm] {$(1,2^\alpha-1)$};
\draw [fill=black] (5.,3.) circle (1.5pt);
\draw [fill=black] (7.,3.) circle (1.5pt);
\draw [fill=black] (9.,3.) circle (1.5pt);
\draw [fill=black] (5.,2.) circle (1.5pt);
\draw [fill=black] (7.,2.) circle (1.5pt);
\draw [fill=black] (9.,2.) circle (1.5pt);
\draw [fill=black] (-1.,3.) circle (1.5pt) node[below,yshift=-0.15cm] {$(1,2^{\alpha-1}-1)$} ;
\draw [fill=black] (-3.,3.) circle (1.5pt) node[below,yshift=-0.15cm] {$(1,1)$};
\draw [fill=black] (6.,3.) circle (1.5pt);
\draw [fill=black] (6.,2.) circle (1.5pt);
\draw[decoration={brace,mirror},decorate] (5,1.75) --(9,1.75) node[below,xshift=-2cm,yshift=-0.25cm] {$\dfrac{\phi(n)-4}{2}$};
\end{scriptsize}
\end{tikzpicture}
\caption{$Cl_2(\mathbb{Z}_{2^\alpha})$}
\end{figure}
\begin{lem}\label{dimfi}
	Let $n=p^\alpha q^\beta$, where $p,q$ are primes and $\alpha, \beta$ are  positive integers. Then \begin{align*}
		SO(Cl_2(\mathbb{Z}_{p^\alpha q^\beta})) &=\dfrac{3}{2}\sqrt{2}(\phi(n)-r)+2r\sqrt{4+(1+\phi(n))^2}\\
		&+2(\phi(n)-r)\sqrt{9+(2+\phi(n))^2}\\
		&+\sqrt{2}(\phi(n)-r)(2+\phi(n))(\phi(n)-r+1)\\
		&+r^2\sqrt{2}(1+\phi(n))+2r(\phi(n)-r)\sqrt{2\phi(n)^2+6\phi(n)+5}.
	\end{align*}
\end{lem}
\begin{proof}
	{By Lemma \ref{dim1},$|Id^\ast(\mathbb{Z}_n)|=2^2-1=3$, so let  
		$Id^\ast(\mathbb{Z}_n)=\{e_1=1,e_2,e_3\}$. By Lemma \ref{dim2}, $\mathbb{Z}_n$ has $r$ self-inverse elements, where $r\in \{2,4,8\}.$ One may partitioned $V(Cl_2(\mathbb{Z}_n))$ as follows:
		
		\[V(Cl_2(\mathbb{Z}_n))=V_1\sqcup V_2\sqcup V_3,\]

where 

\begin{align*}
	&V_1=\{(1,1),(1,u_2),\ldots,(1,u_r),(1,u_{r+1}),\ldots,(1,u_{\phi(n)})\},\\
	&V_2=\{(e_2,1),(e_2,u_2),\ldots,(e_2,u_r),(e_2,u_{r+1}),\ldots,(e_2,u_{\phi(n)})\}~and \\
	&V_3=\{(e_3,1),(e_3,u_2),\ldots,(e_3,u_r),(e_3,u_{r+1}),\ldots,(e_3,u_{\phi(n)})\}.
\end{align*} 
{\bf We note that 
	
	• for every vertex	of the form $(e_i,u_i)$, the component $u_i$ is self-inverse
	
	•• $e_3=1-e_2$ and so $e_2e_3=0$. Thus every vertex in $V_2$ is adjacent to all vertices in $V_3$ (and vice versa).

••• For convenience, a vertex of the form $(a,b)$ is called self-inverse if $b$ is self-inverse.}

The following table illustrates how vertices connect to each other
\begin{table}[H]
\centering
\caption{}\label{ta1}
\resizebox{\textwidth}{!}{\begin{tabular}{|c|c|c|c|c|c!{\vrule width 4pt}c|c|c|}
\hline
$1$ & $V_1$ & $(1,1)$ & $(1,u_2)$ & $\cdots$ & $(1,u_r)$ & $(1,u_{r+1})$ & $\cdots$ & $(1,u_{\phi(n)})$\bigstrut\\
\hline
$2$ & Between $V_1$ and $V_1$ & $0$ & $0$ & $0$ & $0$ & $1$ & $1$ & $1$\bigstrut\\
\hline
$3$ & Between $V_1$ and $V_2$ & $1$ & $1$ & $1$ & $1$ & $1$ & $1$ & $1$\bigstrut\\
\hline
$4$ & Between $V_1$ and $V_3$ & $1$ & $1$ & $1$ & $1$ & $1$ & $1$ & $1$\bigstrut\\
\hline
$5$ & deg($v$)
 & $2$ & $2$ & $2$ & $2$ & $3$ & $3$ & $3$\bigstrut\\
\noalign{\hrule height 4pt}
$6$ & $V_2$ & $(e_2,u_1)$ & $(e_2,u_2)$ & $\cdots$ & $(e_2,u_r)$ & $(e_2,u_{r+1})$ & $\cdots$ & $(e_2,u_{\phi(n)})$\bigstrut\\
\hline
$7$ & Between $V_2$ and $V_1$ & $1$ & $1$ & $1$ & $1$ & $1$ & $1$
 & $1$\bigstrut\\
\hline
$8$ & Between $V_2$ and $V_2$ & $\phi(n)$ & $\phi(n)$ & $\phi(n)$ & $\phi(n)$ & $\phi(n)$ & $\phi(n)$ & $\phi(n)$\bigstrut\\
\hline
$9$ & Between $V_2$ and $V_3$ & $0$ & $0$ & $0$ & $0$ & $1$ & $1$ & $1$\bigstrut\\
\hline
$10$ & deg($v$)
 & $1+\phi(n)$ & $1+\phi(n)$ & $1+\phi(n)$ & $1+\phi(n)$ & $2+\phi(n)$ & $2+\phi(n)$ & $2+\phi(n)$\bigstrut\\
\noalign{\hrule height 4pt}
$11$ & $V_3$ & $(e_3,u_1)$ & $(e_3,u_2)$ & $\cdots$ & $(e_3,u_r)$ & $(e_3,u_{r+1})$ & $\cdots$ & $(e_3,u_{\phi(n)})$\bigstrut\\
\hline
$12$ & Between $V_3$ and $V_1$ & $1$ & $1$ & $1$ & $1$ & $1$ & $1$
 & $1$\bigstrut\\
\hline
$13$ & Between $V_3$ and $V_2$ & $\phi(n)$ & $\phi(n)$ & $\phi(n)$ & $\phi(n)$ & $\phi(n)$ & $\phi(n)$ & $\phi(n)$\bigstrut\\
\hline
$14$ & Between $V_3$ and $V_3$ & $0$ & $0$ & $0$ & $0$ & $1$ & $1$ & $1$\bigstrut\\
\hline
$15$ & deg($v$)
 & $1+\phi(n)$ & $1+\phi(n)$ & $1+\phi(n)$ & $1+\phi(n)$ & $2+\phi(n)$ & $2+\phi(n)$ & $2+\phi(n)$\bigstrut\\
\hline
\end{tabular}}
\end{table}
Now, we compute the degree of each vertex.

{\bf(1) $N(V_1)$:} Every vertex in $\{(1,1),(1,u_2),\ldots,(1,u_r)\}$ is self-inverse and so there is no edges between them. Moreover, for every $(1,u_i)\in V_1$ with $r+1\leq i\leq \phi(n)$ there exists  $(1,u_j)\in V_1$ with $r+1\leq j\leq \phi(n)$ such that $u_i\cdot u_j=1$. Thus we have the second row in the Table 1.

Next, we investigate edges between $V_1$ and $V_2$. For every $(1,u)\in V_1$  there exists exactly one $(e_2,v)\in V_2$ such that $u\cdot v=1$. Thus the third and seventh rows of the table are gained. A similar argument gives fourth and twelfth rows of the table.

{\bf(2) $N(V_2)$:} In (1), we investigated edges between $V_1$ and $V_2$. Thus it is enough to consider edges whose end vertices are contained in $V_2$ and edges between $V_2$ and $V_3$. We know that $e_3=1-e_2$ and so $e_2\cdot e_3=0$. Thus every $(e_2,u)\in V_2$ is adjacent to every 
$(e_3,v)\in V_3$. Hence the eighth and thirteenth rows of the table are obtained.

Now, all edges with end vertices  in $V_2$ are considered. Every vertex in \linebreak  $\{(e_2,1),(e_2,u_2),\ldots,(e_2,u_r)\}$ is self-inverse and so there is no edges between them.  Moreover, for every $(e_2,u_i)\in V_2$,
$r+1\leq i\leq \phi(n)$ there exists $(e_2,u_j)\in V_2$, 
$r+1\leq j\leq \phi(n)$, such that
 $u_i\cdot u_j=1$. Therefore, we have the ninth rows of the table and in the tenth row the degree of vertices contained in $V_2$ are determined.
 
 {\bf(2) $N(V_3)$:} By a similar argument to that of $N(V_2)$, fifteenth row is obtained.
 
 Now, we are ready to compute the sombor index of $Cl_2(\mathbb{Z}_{p^\alpha q^\beta})$. Let  $uv\in E(G)$ and consider the following cases:
 
{\bf Case (1) $u,v\in V_1$:} We know that the $r$ first elements in $V_1$ are self-inverse and so there is no edges between them and for every $(1,u_i)\in V_1$ there exists  $(1,u_j)\in V_1$ with $r+1\leq i,j\leq \phi(n)$ such that $u_i\cdot u_j=1$. So there are $\dfrac{\phi(n)-r}{2}$ pairs of adjacent vertices (i.e., we have $\dfrac{\phi(n)-r}{2}$ edges) of degree 3. Hence, the sombor index of these edges is $\dfrac{\phi(n)-r}{2}\sqrt{3^2+3^2}=\dfrac{3\sqrt{2}}{2}(\phi(n)-r)$.

{\bf Case (2) $u\in V_1, v\in V_2$:}  Every self-inverse element of $V_1$ is adjacent one of self-inverse elements of $V_2$. Indeed, if $(1,u_i)\in V_1, (e_2,u_i)\in V_2$, with $1\leq i\leq r$, then $u_i\cdot u_i=1$ and so  $(1,u_i), (e_2,u_i)$ are adjacent. These vertices make $r$ edges. By the table, the sombor index of these edges is $r\sqrt{2^2+(1+\phi(n))^2}=r\sqrt{4+(1+\phi(n))^2}$. Similarly, every non self-inverse element of $V_1$ is adjacent one of non self-inverse elements of $V_2$. Indeed,  for every $(1,u_i)\in V_1$, there exists  $(e_2,u_j)\in V_2$, with $r+1\leq i,j\leq \phi(n)$, such that  $u_i\cdot u_j=1$. These vertices make $\phi(n)-r$ edges whose sombor index is $(\phi(n)-r)\sqrt{3^2+(2+\phi(n))^2}$. Therefore, the sombor index of edges in Case (2) is $r\sqrt{4+(1+\phi(n))^2}+(\phi(n)-r)\sqrt{9+(2+\phi(n))^2}$.

{\bf Case (3) $u\in V_1, v\in V_3$:} By a similar way to that of Case 2, we infer that the sombor index of edges in this case is $r\sqrt{4+(1+\phi(n))^2}+(\phi(n)-r)\sqrt{9+(2+\phi(n))^2}$.

{\bf Case (4) $u,v\in V_2$:} We know that the $r$ first elements of $V_2$ are self-inverse and there is no edge between them. Moreover, for every  $(e_2,u_i)\in V_2$ there exists $(e_2,u_j)\in V_2$, with $r+1\leq i,j\leq \phi(n)$, such that $u_i\cdot u_j=1$. These vertices make $\dfrac{\phi(n)-r}{2}$ of edges and their degrees is $2+\phi(n)$, by the table. Thus the sombor index of these edges is \[\dfrac{\phi(n)-r}{2}\sqrt{(2+\phi(n))^2+(2+\phi(n))^2}=\dfrac{\sqrt{2}}{2}(\phi(n)-r)(2+\phi(n)).\] 

{\bf Case (5) $u,v\in V_3$:} A similar argument to Case 4 implies that the sombor index of edges in this case is $\dfrac{\sqrt{2}}{2}(\phi(n)-r)(2+\phi(n))$.

{\bf Case (6) $u\in V_2, v\in V_3$:} For every $(e_2,u)\in V_2$ and $(e_3,v)\in V_3$  we have $e_2\cdot e_3=0$. Hence all vertices of $V_2,V_3$ are adjacent. We continue the proof in the following subcases:

{\bf Subcase (i): $u,v$ are self-inverse of degree $1+\phi(n)$}. As both of $V_2, V_3$ have $r$ such elements, by the table the sombor index in this case is \[r\times r\times\sqrt{(1+\phi(n))^2+(1+\phi(n))^2}=\sqrt{2}r^2(1+\phi(n)).\]

{\bf Subcase (ii): $u$ is self-inverse of degree $1+\phi(n)$ and $v$ is non self-inverse of degree $2+\phi(n)$.} As $V_2, V_3$, have $r$ and $\phi(n)-r$ of such elements, respectively, the sombor index of these edges by the table is 
\[r\times(\phi(n)-r)\sqrt{(1+\phi(n))^2+(2+\phi(n))^2}.\]

{\bf Subcase (iii): $u$ is non self-inverse of degree $(2+\phi(n))$ and $v$ is 
 self-inverse of degree $1+\phi(n)$.}  As $V_2, V_3$, have $ \phi(n)-r$ and $r$ of such elements, respectively, the sombor index of these edges by the table is 
\[r(\phi(n)-r)\sqrt{(1+\phi(n))^2+(2+\phi(n))^2}.\]

{\bf Subase (iv): $u,v$ are non self-inverse of degree $2+\phi(n)$}. As both of $V_2, V_3$ have $\phi(n)-r$ of such elements, by the table the sombor index in this case is

\[(\phi(n)-r)(\phi(n)-r)\sqrt{(2+\phi(n))^2+(2+\phi(n))^2}=\sqrt{2}(\phi(n)-r)^2(2+\phi(n)).\]

By all of the above cases 

\begin{align*}
	SO(Cl_2(\mathbb{Z}_{p^\alpha q^\beta})) &=\dfrac{3}{2}\sqrt{2}(\phi(n)-r)+2r\sqrt{4+(1+\phi(n))^2}\\
	&+2(\phi(n)-r)\sqrt{9+(2+\phi(n))^2}\\
	&+\sqrt{2}(\phi(n)-r)(2+\phi(n))(\phi(n)-r+1)\\
	&+r^2\sqrt{2}(1+\phi(n))+2r(\phi(n)-r)\sqrt{2\phi(n)^2+6\phi(n)+5}.
\end{align*}
}\end{proof}
The next example computes $SO(Cl_2(\mathbb{Z}_{24}))$.
\begin{example}
Let $n=24$. By Lemma \ref{dim2}, $r=8$. Moreover, $\phi(24)=8$. Thus

$SO(Cl_2(\mathbb{Z}_{24})) =\dfrac{3}{2}\sqrt{2}(8-8)+2\times 8\sqrt{4+(1+8)^2}+2(8-8)\sqrt{9+(2+8)^2}
+\sqrt{2}(8-8)(2+8)(8-8+1)+8^2\sqrt{2}(1+8)
+2\times 8(8-8)\sqrt{2\times 8^2+6\times 8+5}
=16\sqrt{85}+576\sqrt{2}.$	
\end{example}
\begin{figure}[H]
\centering
\includegraphics[scale=0.6]{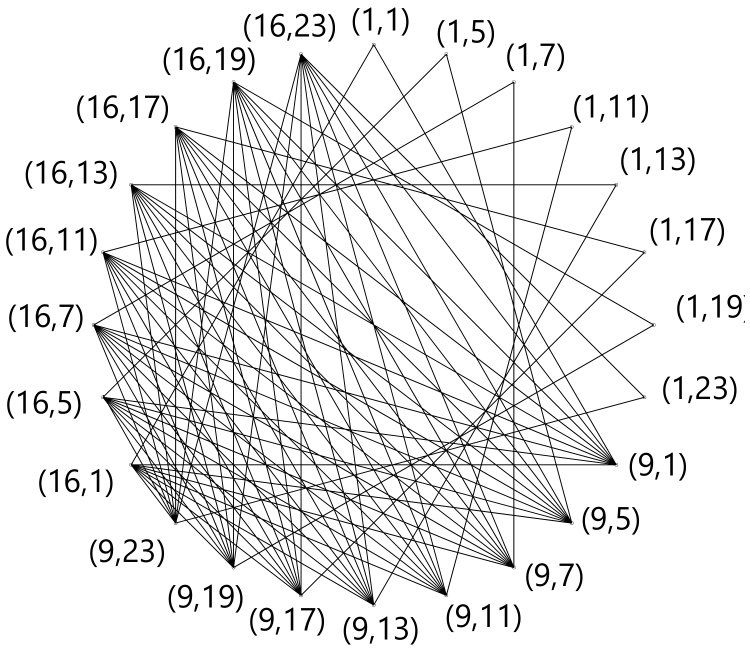}
\caption{$Cl_2(\mathbb{Z}_{24})$}
\end{figure}
\begin{align*}
&V_1=\{(1,1), (1,5), (1,7), (1,11), (1,13), (1,17), (1,19), (1,23)\},\\
&V_2=\{(9,1), (9,5), (9,7), (9,11), (9,13), (9,17), (9,19), (9,23)\},\\
&V_3=\{(16,1), (16,5), (16,7), (16,11), (16,13), (16,17), (16,19), (16,23)\}.
\end{align*}
Now, we are ready to state the main result of this paper.
\begin{thm}\label{akhar}

Let $n=\prod\limits_{i=1}^kp_i^{\alpha_i}$, 
$k\geq 2$. Then \begin{align*}
	SO(Cl_2(\mathbb{Z}_n)) &=\dfrac{\sqrt{2}}{2}(\phi(n)-r)(2^k-1)\\
	&\scalebox{0.75}{$+(2^k-2)[r\times\sqrt{(2^k-2)^2+(2^k-3+\phi(n))^2}+(\phi(n)-r)\sqrt{(2^k-1)^2+(2^k-2+\phi(n))^2}]$}\\
	&+\sqrt{2}r^2(2^{k-1}-1)(2^k-3+\phi(n))\\
	&+2r(\phi(n)-r)(2^{k-1}-1)\sqrt{(2^k-3+\phi(n))^2+(2^k-2+\phi(n))^2}\\
	&+\sqrt{2}(\phi(n)-r)^2(2^{k-1}-1)(2^k-2+\phi(n))\\
	&\scalebox{0.8}{$+(2^{2k-1}-6\times 2^{k-1}+4)[(\sqrt{2}r(2^k-3+\phi(n))+\sqrt{2}(\phi(n)-r)(2^k-2+\phi(n))],$}
\end{align*}
in which 

$$r=|U'(\mathbb{Z}_n)|=\begin{cases}2^s & \text{if} ~ m\in\{0,1\},	\\
	2^{s+1} & \text{if}~m=2, \\
	2^{s+2} & \text{if}~ m\geq3,	
\end{cases}$$

if we write $n=2^m\prod\limits_{i=1}^sp_i^{\alpha_i}$.
\end{thm}

\begin{proof}
{By Lemma \ref{dim1}, $|Id(\mathbb{Z}_n)^\ast|=2^k-1$. For convenience, let $|U'(\mathbb{Z}_n)|=r$. One may consider $V_1\sqcup V_2\sqcup V_3\sqcup\cdots\sqcup V_{2^k-1}$, as partition for $V(Cl_2(\mathbb{Z}_n))$ in a similar way to that of Lemma \ref{dimfi}. To compute the degree of vertices we check the following situations:
\begin{table}[H]
\centering
\caption{}\label{ta2}
\scalebox{0.52}{\begin{tabular}{|c|c|c|c|c|c!{\vrule width 4pt}c|c|c|}
\hline
$1$ & $V_1$ & $(1,1)$ & $(1,u_2)$ & $\cdots$ & $(1,u_r)$ & $(1,u_{r+1})$ & $\cdots$ & $(1,u_{\phi(n)})$\\
\hline
$2$ & Between $V_1$ and $V_1$ & $0$ & $0$ & $0$ & $0$ & $1$ & $1$
 & $1$\\
\hline
$3$ & Between $V_1$ and $V_i$ & $2^k-2$ & $2^k-2$ & $2^k-2$ & $2^k-2$ & $2^k-2$
 & $2^k-2$ & $2^k-2$\\
\hline
$4$ & deg($v$)
 & $2^k-2$ & $2^k-2$ & $2^k-2$ & $2^k-2$ & $2^k-1$ & $2^k-1$
  & $2^k-1$\\
\noalign{\hrule height 4pt}
$5$ & $V_2$ & $(e_2,1)$ & $(e_2,u_2)$ & $\cdots$ & $(e_2,u_r)$ & $(e_2,u_{r+1})$ & $\cdots$ & $(e_2,u_{\phi(n)})$\\
\hline
$6$ & Between $V_2$ and $V_1$ & $1$ & $1$ & $1$ & $1$ & $1$ & $1$
 & $1$\\
\hline
$7$ & Between $V_2$ and $V_2$ & $0$ & $0$ & $0$ & $0$ & $1$ & $1$
 & $1$\\
\hline
$8$ & Between $V_2$ and $V_3$ & $\phi(n)$ & $\phi(n)$ & $\phi(n)$ & $\phi(n)$ & $\phi(n)$ & $\phi(n)$ & $\phi(n)$\\
\hline
$9$ &($i\neq 1,2,3$) Between $V_2$ and $V_i$ & $2^k-4$ & $2^k-4$ & $2^k-4$ & $2^k-4$ & $2^k-4$ & $2^k-4$  & $2^k-4$\\
\hline
$10$ & deg($v$)
 & $2^k-3+\phi(n)$ & $2^k-3+\phi(n)$ & $2^k-3+\phi(n)$ & $2^k-3+\phi(n)$ & $2^k-2+\phi(n)$ & $2^k-2+\phi(n)$ & $2^k-2+\phi(n)$\\
\hline
$11$ & $V_3$ & $(e_3,1)$ & $(e_3,u_2)$ & $\cdots$ & $(e_3,u_r)$ & $(e_3,u_{r+1})$ & $\cdots$ & $(e_3,u_{\phi(n)})$\\
\hline
$12$ & Between $V_3$ and $V_1$ & $1$ & $1$ & $1$ & $1$ & $1$ & $1$
 & $1$\\
\hline
$13$ & Between $V_3$ and $V_3$ & $0$ & $0$ & $0$ & $0$ & $1$ & $1$
 & $1$\\
\hline
$14$ & Between $V_3$ and $V_2$ & $\phi(n)$ & $\phi(n)$ & $\phi(n)$ & $\phi(n)$ & $\phi(n)$ & $\phi(n)$ & $\phi(n)$\\
\hline
$15$ &($i\neq 1,2,3$) between $V_3$ and $V_i$ & $2^k-4$ & $2^k-4$ & $2^k-4$ & $2^k-4$ & $2^k-4$ & $2^k-4$  & $2^k-4$\\
\hline
$16$ & deg($v$)
 & $2^k-3+\phi(n)$ & $2^k-3+\phi(n)$ & $2^k-3+\phi(n)$ & $2^k-3+\phi(n)$ & $2^k-2+\phi(n)$ & $2^k-2+\phi(n)$ & $2^k-2+\phi(n)$\\
\hline
\noalign{\hrule height 4pt}
$17$ & $V_4$ & $(e_4,1)$ & $(e_4,u_2)$ & $\cdots$ & $(e_4,u_r)$ & $(e_4,u_{r+1})$ & $\cdots$ & $(e_4,u_{\phi(n)})$\\
\hline
$18$ & Between $V_4$ and $V_1$ & $1$ & $1$ & $1$ & $1$ & $1$ & $1$
 & $1$\\
\hline
$19$ & Between $V_4$ and $V_4$ & $0$ & $0$ & $0$ & $0$ & $1$ & $1$
 & $1$\\
\hline
$20$ & Between $V_4$ and $V_5$ & $\phi(n)$ & $\phi(n)$ & $\phi(n)$ & $\phi(n)$ & $\phi(n)$ & $\phi(n)$ & $\phi(n)$\\
\hline
$21$ &($i\neq 1,4,5$) Between $V_4$ and $V_i$ & $2^k-4$ & $2^k-4$ & $2^k-4$ & $2^k-4$ & $2^k-4$ & $2^k-4$  & $2^k-4$\\
\hline
$22$ & deg($v$)
 & $2^k-3+\phi(n)$ & $2^k-3+\phi(n)$ & $2^k-3+\phi(n)$ & $2^k-3+\phi(n)$ & $2^k-2+\phi(n)$ & $2^k-2+\phi(n)$ & $2^k-2+\phi(n)$\\
\hline
$23$ & $V_5$ & $(e_5,1)$ & $(e_5,u_2)$ & $\cdots$ & $(e_5,u_r)$ & $(e_5,u_{r+1})$ & $\cdots$ & $(e_5,u_{\phi(n)})$\\
\hline
$24$ & Between $V_5$ and $V_1$ & $1$ & $1$ & $1$ & $1$ & $1$ & $1$
 & $1$\\
\hline
$25$ & Between $V_5$ and $V_5$ & $0$ & $0$ & $0$ & $0$ & $1$ & $1$
 & $1$\\
\hline
$26$ & Between $V_5$ and $V_4$ & $\phi(n)$ & $\phi(n)$ & $\phi(n)$ & $\phi(n)$ & $\phi(n)$ & $\phi(n)$ & $\phi(n)$\\
\hline
$27$ &($i\neq 1,4,5$) between $V_5$ and $V_i$ & $2^k-4$ & $2^k-4$ & $2^k-4$ & $2^k-4$ & $2^k-4$ & $2^k-4$  & $2^k-4$\\
\hline
$28$ & deg($v$)
 & $2^k-3+\phi(n)$ & $2^k-3+\phi(n)$ & $2^k-3+\phi(n)$ & $2^k-3+\phi(n)$ & $2^k-2+\phi(n)$ & $2^k-2+\phi(n)$ & $2^k-2+\phi(n)$\\
\hline
\noalign{\hrule height 4pt}
 & $\cdots$ & $\cdots$ & $\cdots$ & $\cdots$ & $\cdots$ & $\cdots$ & $\cdots$ & $\cdots$\\
\noalign{\hrule height 4pt}
$29$ & $V_{2^k}-2$ & $(e_{2^k}-2,1)$ & $(e_{2^k}-2,u_2)$ & $\cdots$ & $(e_{2^k}-2,u_r)$ & $(e_{2^k}-2,u_{r+1})$ & $\cdots$ & $(e_{2^k}-2,u_{\phi(n)})$\\
\hline
$30$ & Between $V_{2^k}-2$ and $V_1$ & $1$ & $1$ & $1$ & $1$ & $1$ & $1$
 & $1$\\
\hline
$31$ & Between $V_{2^k}-2$ and $V_{2^k}-2$ & $0$ & $0$ & $0$ & $0$ & $1$ & $1$
 & $1$\\
\hline
$32$ & Between $V_{2^k}-2$ and $V_{2^k}-1$ & $\phi(n)$ & $\phi(n)$ & $\phi(n)$ & $\phi(n)$ & $\phi(n)$ & $\phi(n)$ & $\phi(n)$\\
\hline
$33$ &($i\neq 1,{2^k}-2,{2^k}-1$) Between $V_{2^k}-2$ and $V_i$ & $2^k-4$ & $2^k-4$ & $2^k-4$ & $2^k-4$ & $2^k-4$ & $2^k-4$  & $2^k-4$\\
\hline
$34$ & deg($v$)
 & $2^k-3+\phi(n)$ & $2^k-3+\phi(n)$ & $2^k-3+\phi(n)$ & $2^k-3+\phi(n)$ & $2^k-2+\phi(n)$ & $2^k-2+\phi(n)$ & $2^k-2+\phi(n)$\\
\hline
$35$ & $V_{2^k}-1$ & $(e_{2^k}-1,1)$ & $(e_{2^k}-1,u_2)$ & $\cdots$ & $(e_{2^k}-1,u_r)$ & $(e_{2^k}-1,u_{r+1})$ & $\cdots$ & $(e_{2^k}-1,u_{\phi(n)})$\\
\hline
$36$ & Between $V_{2^k}-1$ and $V_1$ & $1$ & $1$ & $1$ & $1$ & $1$ & $1$
 & $1$\\
\hline
$37$ & Between $V_{2^k}-1$ and $V_{2^k}-1$ & $0$ & $0$ & $0$ & $0$ & $1$ & $1$
 & $1$\\
\hline
$38$ & Between $V_{2^k}-1$ and $V_{2^k}-2$ & $\phi(n)$ & $\phi(n)$ & $\phi(n)$ & $\phi(n)$ & $\phi(n)$ & $\phi(n)$ & $\phi(n)$\\
\hline
$39$ &($i\neq 1,{2^k}-2,{2^k}-1$) Between $V_{2^k}-1$ and $V_i$ & $2^k-4$ & $2^k-4$ & $2^k-4$ & $2^k-4$ & $2^k-4$ & $2^k-4$  & $2^k-4$\\
\hline
$40$ & deg($v$)
 & $2^k-3+\phi(n)$ & $2^k-3+\phi(n)$ & $2^k-3+\phi(n)$ & $2^k-3+\phi(n)$ & $2^k-2+\phi(n)$ & $2^k-2+\phi(n)$ & $2^k-2+\phi(n)$\\
\hline
\noalign{\hrule height 4pt}
\end{tabular}}
\end{table}
{\bf(1) $N(V_1)$:} In a similar way to that of $n=p^\alpha q^\beta$, we get the second row, which shows that each vertex of $V_1$ has how many neighbors in $V_1$. On the other hand, for every $(1,u)\in V_1$, there exists  $(e,v)\in V_i$, $2\leq i \leq 2^k-1$, such that  $u\cdot v=1$. Hence, the third, sixth, twelfth, eighteenth, twenty-forth, thirtieth and thirty-sixth rows are obtained. We have inserted the degree of each vertex contained in $V_1$ in the fourth row.  

{\bf(2) $N(V_2)$:} In a similar way to that of (1), we get the seventh row of the table, which shows that each vertex of $V_2$ has how many neighbors in $V_2$. Next, we find the neighbors of  vertices of $V_2$ in $V_3$. We know that each vertex $(e_2,u)\in V_2$ is adjacent to $(e_3,v)\in V_3$. Therefore, we get the eighth and fourteenth rows of the table. Now, we calculate the number of neighbors in $V_i$ 
($4\leq i\neq 2^k-1$), for every vertex in  $V_2$. For every vertex $(e_2,u)\in V_2$ there exists $(e_i,v)\in V_i$ ($4\leq i\neq 2^k-1$) such that $u\cdot v=1$. Thus, in the tenth row of the table we have the degree of each vertex of $V_2$. 

{\bf(3) $N(V_3)$:} It is similar to (2) .

{\bf(4)} In a similar manner, we get the degrees all vertices. If $v$ is an arbitrary vertex, then $deg(v)\in \{2^k-2, 2^k-1, 2^k-3+\phi(n),  2^k-2+\phi(n)\}$.

 Now, we are ready to compute the sombor index of $Cl_2(\mathbb{Z}_{n})$. Let  $uv\in E(G)$ and consider the following cases:

{\bf Case (1) $u,v\in V_1$:} There are $\dfrac{\phi(n)-r}{2}$ edges with end vertices in $V_1$, each of ends has degree  $2^k-1$. Thus the sombor index of these edges is 
\[\dfrac{\phi(n)-r}{2}\sqrt{(2^k-1)^2+(2^k-1)^2}=\dfrac{\sqrt{2}}{2}(\phi(n)-r)(2^k-1).\]

{\bf Case (2) $u\in V_1$, $v\in V_i$
	$i=2,3,\ldots,2^k-1$:}
Every self-inverse vertex in $V_1$ of degree $2^k-2$ is adjacent to a self-inverse vertex in $V_i$ ($i\neq 1$) of degree $2^k-3+\phi(n)$. Moreover, every non self-inverse vertex in  $V_1$ which has degree $2^k-1$ is adjacent to (exactly) one non self-inverse vertex in $V_i$ ($i\neq 1$) of degree $2^k-2+\phi(n)$. We know that each $V_i$ ($i=1,2,\ldots,2^k-1$) has $r$ and $\phi(n)-r$ self-inverse and non self-inverse elements, respectively. Thus the sombor index of edges between $V_1$ and each   $V_i$ ($i=2,3,\ldots,2^k-1$) is \[r\times\sqrt{(2^k-2)^2+(2^k-3+\phi(n))^2}+(\phi(n)-r)\times\sqrt{(2^k-1)^2+(2^k-2+\phi(n))^2}.\] Since there are  $2^k-2$ of such $V_i$'s, the total sombor index in this case is 
\[\scalebox{0.85}{$(2^k-2)(r\times\sqrt{(2^k-2)^2+(2^k-3+\phi(n))^2}+(\phi(n)-r)\times\sqrt{(2^k-1)^2+(2^k-2+\phi(n))^2}).$}\]

{\bf Case (3) $u\in V_{2m}$, $v\in V_{2m+1}$, $m=1,2,\ldots,\dfrac{2^k-4}{2},\dfrac{2^k-2}{2}$:} Each vertex in $V_{2m}$ is adjacent to every vertex in $V_{2m+1}$ ($m=1,2,\ldots,\dfrac{2^k-4}{2},\dfrac{2^k-2}{2}$). We consider the following four subcases:

{\bf Subcase (i): Both of end vertices in $V_{2m}$ and $V_{2m+1}$ are self-inverse of degrees $2^k-3+\phi(n)$ ($m=1,2,\ldots,\dfrac{2^k-4}{2},\dfrac{2^k-2}{2}$)}. There are $r$ self-inverse elements in each $V_i$ ($i=2,3,\ldots,2^k-1$) and all of them are adjacent to each other. Moreover, there are $\dfrac{2^k-2}{2}$ possibilities to choose $V_i$'s. Thus the sombor index in this subcase is \[\scalebox{0.85}{$r\times r\times(\dfrac{2^k-2}{2})\sqrt{(2^k-3+\phi(n))^2+(2^k-3+\phi(n))^2}=\sqrt{2}r^2(2^{k-1}-1)(2^k-3+\phi(n)).$}\]

{\bf Subcase (ii): One end is a self-inverse vertex in $V_{2m}$  of degree  $2^k-3+\phi(n)$ and the other end is a non self-inverse vertex in $V_{2m+1}$  of degree $2^k-2+\phi(n)$}. There are $r$ self-inverse and $\phi(n)-r$ non self-inverse vertices in each  $V_{i}$, respectively. Moreover, there are $2^{k-1}-1$  possibilities to choose  $V_i$'s. Thus the sombor index in this subcase is \[r\times(\phi(n)-r)(2^{k-1}-1)\sqrt{(2^k-3+\phi(n))^2+(2^k-2+\phi(n))^2}.\]

{\bf Subcase (iii): One end is a non self-inverse vertex in $V_{2m}$  of degree  $2^k-2+\phi(n)$ and the other end is a self-inverse vertex in $V_{2m+1}$  of degree $2^k-3+\phi(n)$}.
	By a similar argument to that of Subcase (ii), the sombor index is 

\[(\phi(n)-r)\times r\times(2^{k-1}-1)\sqrt{(2^k-2+\phi(n))^2+(2^k-3+\phi(n))^2}.\]

{\bf Subcase (iv):  Both of end vertices in $V_{2m}$ and $V_{2m+1}$ are non self-inverse of degrees $2^k-2+\phi(n)$.} The sombor index in this case is 

\begin{align*}
&(\phi(n)-r)(\phi(n)-r)(2^{k-1}-1)\sqrt{(2^k-2+\phi(n))^2
	+(2^k-2+\phi(n))^2}\\
&=\sqrt{2}(\phi(n)-r)^2(2^{k-1}-1)(2^k-2+\phi(n)).\end{align*}

{\bf Case (4) $u\in V_i$, $v\in V_j$, $i=2,3,\ldots,2^k-1$,
$i\neq j$:} If  $u\in V_{2m}$, $v\notin V_{2m+1}$, $m=1,2,\ldots,\dfrac{2^k-4}{2},\dfrac{2^k-2}{2}$. For instance, we suppose that one end is in $V_2$ and the other one is in  $V_4$. Then we multiple the result to all possible cases. For every self-inverse  $(e_2,u)\in V_2$ of degree   $2^k-3+\phi(n)$ there exists  $(e_4,v)\in V_4$  of degree $2^k-3+\phi(n)$ such that  $u\cdot v=1$. There are $r$ self-inverse elements in each $V_i$, and so the sombor index of such edges is 
\[r\sqrt{(2^k-3+\phi(n))^2+(2^k-3+\phi(n))^2}=\sqrt{2}r(2^k-3+\phi(n)).\]
In a same manner, for every non self-inverse element $(e_2,u)\in V_2$ of degree $2^k-2+\phi(n)$ there exists a non self-inverse element $(e_4,v)\in V_4$ of degree  $2^k-2+\phi(n)$ such that they are adjacent. Since the number of such elements in each $V_i$ is $\phi(n)-r$ , the sombor index of such edges is 
\[(\phi(n)-r)\sqrt{(2^k-2+\phi(n))^2+(2^k-2+\phi(n))^2}=\sqrt{2}(\phi(n)-r)(2^k-2+\phi(n)).\] Finally, the sombor index of all edges with end vertices in $V_2,V_4$ is 
\[\sqrt{2}r(2^k-3+\phi(n))+\sqrt{2}(\phi(n)-r)(2^k-2+\phi(n)).\]

We note that each edge has one end in $V_p$ and one end in $V_q$ such that ($p\neq q$, $p,q\neq 1$). Obviously, the number of all choices is $\binom{2^k-2}{2}$, but we investigated  $V_{2m}$,  $V_{2m+1}$ ($m=1,2,\ldots,\dfrac{2^k-4}{2},\dfrac{2^k-2}{2}$) (there are $2^{k-1}-1$ possible cases). Thus there are 
\[\binom{2^k-2}{2}-(2^{k-1}-1)=2^{2k-1}-6\times 2^{k-1}+4\] desired possible cases. Hence the sombor index in this case is

\[(2^{2k-1}-6\times 2^{k-1}+4)[(\sqrt{2}r(2^k-3+\phi(n))+\sqrt{2}(\phi(n)-r)(2^k-2+\phi(n))].\]

Therefore, the sombor index of $Cl_2(\mathbb{Z}_n)$, where $n=\prod\limits_{i=1}^kp_i^{\alpha_i}$,
$k\geq 2$, is 
\begin{align*}
	SO(Cl_2(\mathbb{Z}_n)) &=\dfrac{\sqrt{2}}{2}(\phi(n)-r)(2^k-1)\\
	&\scalebox{0.75}{$+(2^k-2)[r\times\sqrt{(2^k-2)^2+(2^k-3+\phi(n))^2}+(\phi(n)-r)\sqrt{(2^k-1)^2+(2^k-2+\phi(n))^2}]$}\\
	&+\sqrt{2}r^2(2^{k-1}-1)(2^k-3+\phi(n))\\
	&+2r(\phi(n)-r)(2^{k-1}-1)\sqrt{(2^k-3+\phi(n))^2+(2^k-2+\phi(n))^2}\\
	&+\sqrt{2}(\phi(n)-r)^2(2^{k-1}-1)(2^k-2+\phi(n))\\
	&\scalebox{0.8}{$+(2^{2k-1}-6\times 2^{k-1}+4)[(\sqrt{2}r(2^k-3+\phi(n))+\sqrt{2}(\phi(n)-r)(2^k-2+\phi(n))],$}
\end{align*}
in which 

$$r=|U'(\mathbb{Z}_n)|=\begin{cases}2^s & \text{if} ~ m\in\{0,1\},	\\
	2^{s+1} & \text{if}~m=2, \\
	2^{s+2} & \text{if}~ m\geq3,	
\end{cases}$$

if we write $n=2^m\prod\limits_{i=1}^sp_i^{\alpha_i}$.
}
\end{proof}	

The last example of this paper uses Theorem \ref{akhar} to compute $SO(Cl_2(\mathbb{Z}_{30}))$.
\begin{example}
Let $n=30$. Then $r=8$ and 
\[\phi(n)=(2^1-2^0)(3^1-3^0)(5^1-5^0)=1\times 2\times 4=8, \quad k=3.\]
Thus \begin{align*}
	SO(Cl_2(\mathbb{Z}_{30})) &=\dfrac{\sqrt{2}}{2}(8-4)(2^3-1)\\
	&+(2^3-2)[4\times\sqrt{(2^3-2)^2+(2^3-3+8)^2}\\
	&+(8-4)\sqrt{(2^3-1)^2+(2^3-2+8)^2}]\\
	&+\sqrt{2}\times 4^2(2^{3-1}-1)(2^3-3+8)+\\
	&2\times 4(8-4)(2^{3-1}-1)\sqrt{(2^3-3+8)^2+(2^3-2+8)^2}\\
	&+\sqrt{2}(8-4)^2(2^{3-1}-1)(2^3-2+8)\\
	&\scalebox{0.9}{$+(2^{2\times 3-1}-6\times 2^{3-1}+4)[\sqrt{2}\times 4(2^3-3+8)+\sqrt{2}(8-4)(2^3-2+8)]$}\\
	&\scalebox{0.9}{$=14\sqrt{2}+24\sqrt{205}+24\sqrt{245}+624\sqrt{2}+96\sqrt{365}+672\sqrt{2}+1296\sqrt{2}$}\\
	&=2606\sqrt{2}+28\sqrt{5}+24\sqrt{205}+96\sqrt{365}
\end{align*}	
\end{example}	
{\bf Declarations}

\textbf{Conflicts of interest/Competing interests:} There is no conflict of interest regarding the publication of this paper.

\textbf{Data availability:} Not
applicable.

\textbf{Funding:} Not applicable.
%%%%%%%%%%%%%%%%%%%%%%%%%%%%%%%%%%%%%%%%%%%%%%%%%%%%%%%%%%%%%%%%%%%%%%%%%%%%%%%%%%%%%%%%%%%%%%%%%%%%%%%%%%%%%%%%%%%%%%%%%%%%%%%%%%%%%%%%%%%%%%%%%%%%%%

{}

\end{document}